\documentclass[11pt]{article}
 \usepackage[top=1in,bottom=1in,left=1.5in,right=1.5in]{geometry}
\usepackage{amsmath,amssymb}
\usepackage{tikz}
\usetikzlibrary{decorations.pathreplacing}
\usetikzlibrary{calc}
\usepackage{setspace}
\usepackage{amsfonts} 
\usepackage{setspace}
\usepackage{subfigure}
\usepackage{url}
\usepackage{booktabs}
\usepackage{ifthen}
\usepackage{enumerate}

  \newcommand{\Z}{\mathbb{Z}}
  \renewcommand{\i}{\mathbf{i}}
  \renewcommand{\j}{\mathbf{j}}

  \newcommand{\qed}{\hspace*{\fill} $\Box$}

  \newtheorem{theo}{Theorem}[section]
  \newtheorem{defn}[theo]{Definition}
  \newtheorem{prop}[theo]{Proposition}
  \newtheorem{lemma}[theo]{Lemma}

  \newtheorem{example}[theo]{Example}

  \numberwithin{equation}{section} 
  
\begin{document}
  
\title{On the edge-balanced index sets of product graphs}

\author{
Elliot Krop\thanks{Department of Mathematics, Clayton State University, (ElliotKrop@clayton.edu)}
\and 
Sin-Min Lee\thanks{Department of Computer Science, San Jose State University, (lee.sinmin35@gmail.com)}
\and 
Christopher Raridan\thanks{Department of Mathematics, Clayton State University, (ChristopherRaridan@clayton.edu).}
}
   
\date{\today}

\maketitle

\begin {abstract}
We characterize strongly edge regular product graphs and find the edge-balanced index sets of complete bipartite graphs without a perfect matching, the direct product $K_n\times K_2$. We also prove a lemma that is helpful to determine the edge-balanced index sets of regular graphs.
\\[\baselineskip] 2000 Mathematics Subject Classification: 05C78, 05C25
\\[\baselineskip] Keywords: Edge-labeling, partial-coloring, edge-friendly labeling, friendly labeling, cordiality, friendly index set, edge-balanced index set.
\end {abstract}

 \section{Introduction}
 
 \subsection{Definitions}
 
 For basic graph theoretic notation and definitions see Diestel~\cite{Diest}.  All graphs $G(V,E)$ are finite, simple, and undirected with vertex set $V$ and edge set $E$.  The number of vertices is denoted $p(G)$ and the number of edges is $q(G)$. The complete graph on $n$ vertices is denoted $K_n$ and the bipartite graph with $n$ vertices and $m$ vertices in the first and second \emph{part}, respectively, is denoted $K_{n,m}$. We say two vertices $x$ and $y$ are adjacent if there exists an edge $(x,y)$ and write $x \sim y$; if $x$ and $y$ are not adjacent we write $x \not\sim y$.
 
  A \textit{labeling} of a graph $G$ with $H\subseteq G$, is a function $f : H \rightarrow \Z_2$, where $\Z_2 = \{0,1\}$. If $H=E(G)$ $(H=V(G))$ and $f$ is surjective, call the labeling an \emph{edge labeling} (\emph{vertex labeling}). For an edge $e$, let $f(e)$ denote the label on $e$ and call any edge with label $i\in \Z_2$ an $i$-edge. Setting $e_f(i) = \mbox{card}\{e \in E : f(e) = i\}$, we say a labeling $f$ is \emph{edge-friendly} when $|e_f(0)-e_f(1)|\leq 1$. For a vertex $v$, let $N_i(v)=\{u\in V: f(uv)=i\}$ and set $\deg_i(v) = |N_i(v)|$, called the $i$-degree of $v$. An edge-friendly labeling $f : E \rightarrow \Z_2$ induces a partial vertex labeling $f^+ : V \rightarrow \Z_2$ defined by $f^+(v) = 0$ if $\deg_0(v) > \deg_1(v)$, $f^+(v) = 1$ if $\deg_1(v) > \deg_0(v)$, and $f^+(v)$ is not defined if $\deg_0(v) = \deg_1(v)$, that is, $v$ is \emph{unlabeled}. Setting $v_f(i) = \mbox{card}\{v \in V : f^+(v) = i\}$, the \emph{edge-balanced index set} of a graph $G$, $EBI(G)$, is defined as $\{|v_f(0) - v_f(1)| : f \mbox{ is edge-friendly}\}$. A graph $G$ is said to be an \emph{edge-balanced} graph if there is an edge-friendly labeling $f$ satisfying $|v_f(0) - v_f(1)| \leq 1$ and \emph{strongly edge-balanced} if $v_f(0) = v_f(1)$ and $e_f(0) = e_f(1)$.  When the labeling $f$ is clear from context, we write $e(i)$ and $v(i)$.

  \subsection{History}
 
 Binary labelings are a simplification of graceful labelings, introduced by Cahit~\cite{Cahit} in his seminal paper on cordial graphs. Since then, other notions of balance in binary labelings (\cite{LeeLiuTan} and~\cite{KongLee} for example) have been introduced and much work has been done to classify graphs. So far, the problem is far from complete, though there are classifications of index sets for specific graph classes (see~\cite{ChenHuangLeeLiu} and~\cite{ChopraLeeSu} for example).
 
 Our objective is to assign a binary labeling to some substructure of graphs $G$ (e.g., the edges) so that the assignment is balanced and induces a labeling on some other substructure (e.g., the vertices). We then attempt to classify the degree of imbalance in the induced labeling of $G$. We hope that such an index set could form an invariant that in some way can distinguish classes of graphs. In this paper, we focus on the edge-balanced index sets of various product graphs.

 \subsection{Product Graphs}
 
  The following three constructions of graphs were considered in~\cite{Harary}. For further study, see~\cite{ImrichKlavzar}.

\begin{defn}
 The \emph{lexicographic product} (\emph{composition}) of two graphs $G_1(V_1,E_1)$ and $G_2(V_2,E_2)$, denoted by $G_1[G_2]$, is a graph with vertex set $V_1\times V_2$ and edge set $E(G_1[G_2]) = \{((u_1,v_1),(u_2,v_2)) : (u_1,u_2) \in E_1, \mbox{ or } u_1 = u_2 \mbox{ and } (v_1,v_2) \in E_2\}$.
\end{defn}

\begin{defn}
 The \emph{direct product} (or \emph{conjuction}, \emph{Kronecker product}) of two graphs $G_1(V_1,E_1)$ and $G_2(V_2,E_2)$, denoted by $G_1 \times G_2$, is a graph with vertex set $V_1\times V_2$ and edge set $E(G_1 \times G_2) = \{((u_1,v_1),(u_2,v_2)) : (u_1,u_2) \in E_1 \mbox{ and } (v_1,v_2) \in E_2\}$.
\end{defn}

\begin{defn}
 The \emph{Cartesian product} of two graphs $G_1(V_1,E_1)$ and $G_2(V_2,E_2)$, denoted by $G_1 \square G_2$, is a graph with vertex set $V_1 \times V_2$ and edge set $E(G_1 \square G_2) = \{((u_1,v_1),(u_2,v_2)) : v_1=v_2 \mbox{ and } (u_1,u_2) \in E_1, \mbox{ or } u_1 = u_2 \mbox{ and} (v_1,v_2) \in E_2\}$.
\end{defn}

 \section{Strongly Edge Regular Product Graphs}
 
 In~\cite{ChenHuangLeeLiu}, Chen, Huang, Lee, and Liu produced the following result:

\begin{theo} \label{evensize}
If $G$ is a simple connected graph with order $n$ and even size, then there exists an edge-friendly labeling $f$ of $G$ such that $G$ is strongly edge-balanced.
\end {theo}
 
 The next result, found in~\cite{Harary}, can be obtained by simple counting.

\begin{prop} \label{productgraphsize}
Let $G$ and $H$ be finite graphs. Then
\begin{enumerate}
	\item $q(G[H]) = p(H)^2 q(G) + p(G) q(H)$.
	\item $q(G\times H) = 2q(G) q(H)$.
	\item $q(G\square H) = p(H) q(G) + p(G) q(H)$.
\end{enumerate}
\end{prop}

Applying Theorem~\ref{evensize} to Proposition~\ref{productgraphsize} we obtain the following complete characterization of strongly edge-balanced product graphs.

\begin{theo}
Let $G$ and $H$ be finite graphs. Then 
\begin{enumerate}
	\item $G\times H$ is strongly edge-balanced.
	\item $G \square H$ is strongly edge-balanced if and only if $G[H]$ is strongly edge-balanced.
 \item The following conditions imply $G\square H$ and $G[H]$ are strongly edge-balanced and if $G\square H$ or $G[H]$ are strongly edge-balanced, at least one of the following conditions must hold: 		
	\begin{itemize}
		\item $p(G) \equiv p(H) \equiv 0 \pmod 2$.
		\item $q(G) \equiv q(H) \equiv 0 \pmod 2$.
		\item $q(G) \equiv p(G) \equiv 0 \pmod 2$.
		\item $q(H) \equiv p(H) \equiv 0 \pmod 2$.
		\item $p(H) \equiv q(H) \equiv q(G) \equiv p(G) \equiv 1 \pmod 2$.
	\end{itemize} 
\end{enumerate}
\end{theo}

 \section{Regular Graphs}
 
 A common approach to finding the edge-balanced index set is to start by finding the maximum element in the set and then produce the smaller values by a series of edge-label switches. The following lemma gives a general upper bound on the maximum edge-balanced index set of regular graphs. We show that $K_n\times K_2$ attains this upper bound in the next section. 
 
\begin{lemma}\label{maxEBIregular}
 Let $G$ be a graph of order $n$ and regularity $r$, and let $F$ be the set of edge-friendly labelings of $G$. When $r$ is odd and $n \equiv 0 \pmod 4$,
 \[\max_{f\in F}\{EBI(G)\}\leq \frac{(r-1)n}{r+1}\]
 and when $n\equiv 2 \pmod 4$
 \[\max_{f\in F}\{EBI(G)\}\leq \frac{(r-1)n+4}{r+1}\]
 \end{lemma}
 
 \begin{proof}
 Let $G$ be a graph of order $n$ and regularity $r$ and let $f\in F$. Set $D=\{v\in V(G):\deg_0(v)>\frac{r}{2}\}$ and $S=\{v\in V(G):\deg_1(v)>\frac{r}{2}\}$. 
 
Suppose $r$ is odd. For any $f$, $|E|=\frac{rn}{2}$. We first suppose that $n\equiv 0 \pmod 4$ and consequently $e(0)=e(1)$. Since $\min_{v\in D}(\deg_0(v))\leq r$ and $\min_{v\in S}(\deg_1(v))\geq \frac{r+1}{2}$, we have $\max_{v\in S}(\deg_0(v))\leq \frac{r-1}{2}$, so that
\begin{align*}
	\frac{rn}{2}=2e(0)\leq v(0)r+v(1)\frac{r-1}{2} = v(0)\frac{r+1}{2}+n\frac{r-1}{2}
\end{align*}
which implies 
 \[v(0)\geq \frac{n}{r+1}.\]
Hence, for odd $r$, $n\equiv 0 \pmod 4$,
 \[\max_{f\in F}\{EBI(G)\}\leq \frac{(r-1)n}{r+1}.\]
 Next, suppose that $n\equiv 2 \pmod 4$ and consequently $e(0)=e(1)+1$. Since $\min_{v\in D}(\deg_0(v))\leq r$ and $\min_{v\in S}(\deg_1(v))\geq \frac{r+1}{2}$, we have $\max_{v\in S}(\deg_0(v))\leq \frac{r-1}{2}$, so that
\begin{align*}
	\frac{rn}{2}=2e(0)+1\leq v(0)r+v(1)\frac{r-1}{2}+1 = v(0)\frac{r+1}{2}+n\frac{r-1}{2}+1
\end{align*}
which implies 
 \[v(0)\geq \frac{n-2}{r+1}.\]
Hence, for odd $r$, $n\equiv 2 \pmod 4$,
 \[\max_{f\in F}\{EBI(G)\}\leq \frac{(r-1)n+4}{r+1}.\]

\qed \end{proof}

 Both the maximum and minimum numbers of unlabeled vertices in edge-friendly labelings of $K_n$ were found in~\cite{KropLeeRaridan}. For other graphs, this problem is open.

 \section{Bipartite Double Covers}

We define the \emph{bipartite double cover} of $G$ as the direct product $G \times K_2$. The bipartite double cover of the complete graph $K_n$ is called the \emph{crown graph} and is the complete bipartite graph $K_{n,n}$ minus a perfect matching.  Note that $K_n\times K_2$ is a regular graph with odd regularity for even $n$ and even regularity for odd $n$.

\begin{lemma} \label{evenEBI}
If $G$ is a graph with all vertices of odd degree, then $EBI(G)$ contains only even integers.
\end{lemma}
 
\begin{proof}
 Let $G$ be a graph of order $n$ where each vertex is of odd degree. For any edge-friendly labeling of $G$, there are no unlabeled vertices. Hence $v(1)+v(0)=n$ and by the Handshaking Lemma, $n$ is even. Assume without loss of generality that $v(1)\geq v(0)$. If $v(1)-v(0)$ is odd, then $v(1)+v(0)+v(1)-v(0)=2v(1)$ must be odd, which is a contradiction. Thus, $v(1)-v(0)$ is always even. \qed
\end{proof}
 
\begin{example}
 In Figures~\ref{fig:EBI-K3xK2}-~\ref{fig:EBI-K5xK2}, we give $EBI(K_3\times K_2)$, $EBI(K_4\times K_2)$, and $EBI(K_5\times K_2)$, respectively, and provide examples of labelings that produce each index.
\end{example}
 

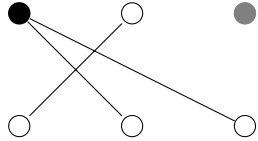
\begin{figure}[ht]
\centering
\begin{tikzpicture}
\tikzstyle{0-label}=[circle,fill=black!50,inner sep=3pt]
\tikzstyle{1-label}=[circle,fill=black,inner sep=3pt]

\node[1-label] (v-1) at (0,1.5) {};

\foreach \i/\j/\k in {1.5/1.5/2,0/0/4,1.5/0/5,3/0/6}
{
  \draw
    (\i,\j) circle (4pt);
  \node[] (v-\k) at (\i,\j) {};
}

\node[0-label] (v-3) at (3,1.5) {};

\draw[color=black] 
(v-1) -- (v-5)
(v-1) -- (v-6)
(v-2) -- (v-4);
\end{tikzpicture}
\caption{$EBI(K_3 \times K_2) = \{ 0 \}$. We show only the $1$-edges.}
\label{fig:EBI-K3xK2}
\end{figure}


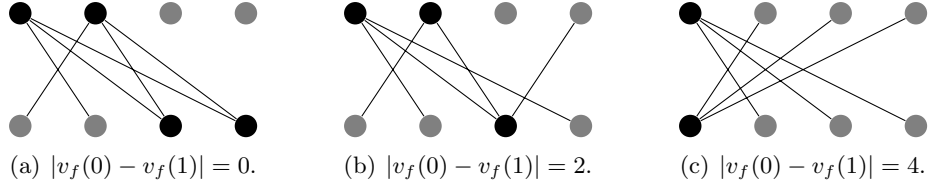
\begin{figure}[ht]
\centering
\subfigure[$|v_f(0)-v_f(1)|=0$.]
{
\begin{tikzpicture}
\tikzstyle{0-label}=[circle,fill=black!50,inner sep=3pt]
\tikzstyle{1-label}=[circle,fill=black,inner sep=3pt]

\foreach \i/\j/\k in {0/1.5/1,1/1.5/2,2/0/7,3/0/8}
{
  \node[1-label] (v-\k) at (\i,\j) {};
}

\foreach \i/\j/\k in {2/1.5/3,3/1.5/4,0/0/5,1/0/6}
{
  \node[0-label] (v-\k) at (\i,\j) {};
}

\draw[color=black] 
(v-1) -- (v-6)
(v-1) -- (v-7)
(v-1) -- (v-8)
(v-2) -- (v-5)
(v-2) -- (v-7)
(v-2) -- (v-8);
\end{tikzpicture}
}
\quad\quad
\subfigure[$|v_f(0)-v_f(1)|=2$.]
{
\begin{tikzpicture}
\tikzstyle{0-label}=[circle,fill=black!50,inner sep=3pt]
\tikzstyle{1-label}=[circle,fill=black,inner sep=3pt]

\foreach \i/\j/\k in {0/1.5/1,1/1.5/2,2/0/7}
{
  \node[1-label] (v-\k) at (\i,\j) {};
}

\foreach \i/\j/\k in {2/1.5/3,3/1.5/4,0/0/5,1/0/6,3/0/8}
{
  \node[0-label] (v-\k) at (\i,\j) {};
}

\draw[color=black] 
(v-1) -- (v-6)
(v-1) -- (v-7)
(v-1) -- (v-8)
(v-2) -- (v-5)
(v-2) -- (v-7)
(v-4) -- (v-7);
\end{tikzpicture}
}
\quad\quad
\subfigure[$|v_f(0)-v_f(1)|=4$.]
{
\begin{tikzpicture}
\tikzstyle{0-label}=[circle,fill=black!50,inner sep=3pt]
\tikzstyle{1-label}=[circle,fill=black,inner sep=3pt]

\foreach \i/\j/\k in {0/1.5/1,0/0/5}
{
  \node[1-label] (v-\k) at (\i,\j) {};
}

\foreach \i/\j/\k in {1/1.5/2,2/1.5/3,3/1.5/4,1/0/6,2/0/7,3/0/8}
{
  \node[0-label] (v-\k) at (\i,\j) {};
}

\draw[color=black] 
(v-1) -- (v-6)
(v-1) -- (v-7)
(v-1) -- (v-8)
(v-5) -- (v-2)
(v-5) -- (v-3)
(v-5) -- (v-4);
\end{tikzpicture}
}
\caption{$EBI(K_4 \times K_2) = \{ 0,2,4 \}$. We show only the $1$-edges.}
\label{fig:EBI-K4xK2}
\end{figure}


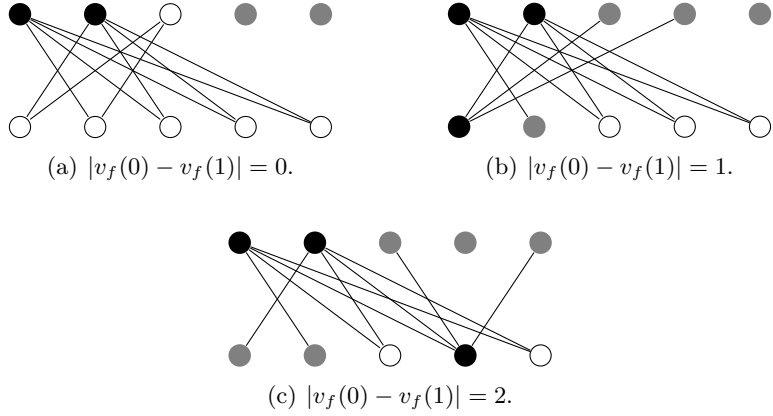
\begin{figure}[ht]
\centering
\subfigure[$|v_f(0)-v_f(1)|=0$.]
{
\begin{tikzpicture}
\tikzstyle{0-label}=[circle,fill=black!50,inner sep=3pt]
\tikzstyle{1-label}=[circle,fill=black,inner sep=3pt]

\foreach \i/\j/\k in {0/1.5/1,1/1.5/2}
{
  \node[1-label] (v-\k) at (\i,\j) {};
}

\foreach \i/\j/\k in {3/1.5/4,4/1.5/5}
{
  \node[0-label] (v-\k) at (\i,\j) {};
}

\foreach \i/\j/\k in {2/1.5/3,0/0/6,1/0/7,2/0/8,3/0/9,4/0/10}
{
  \draw
    (\i,\j) circle (4pt);
  \node[] (v-\k) at (\i,\j) {};
}

\foreach \i/\j in {1/7,1/8,1/9,1/10,2/6,2/8,2/9,2/10,3/6,3/7}
{
\draw[color=black] 
	(v-\i) -- (v-\j);
}
\end{tikzpicture}
}
\quad\quad\quad
\subfigure[$|v_f(0)-v_f(1)|=1$.]
{
\begin{tikzpicture}
\tikzstyle{0-label}=[circle,fill=black!50,inner sep=3pt]
\tikzstyle{1-label}=[circle,fill=black,inner sep=3pt]

\foreach \i/\j/\k in {0/1.5/1,1/1.5/2,0/0/6}
{
  \node[1-label] (v-\k) at (\i,\j) {};
}

\foreach \i/\j/\k in {2/1.5/3,3/1.5/4,4/1.5/5,1/0/7}
{
  \node[0-label] (v-\k) at (\i,\j) {};
}

\foreach \i/\j/\k in {2/0/8,3/0/9,4/0/10}
{
  \draw
    (\i,\j) circle (4pt);
  \node[] (v-\k) at (\i,\j) {};
}

\foreach \i/\j in {1/7,1/8,1/9,1/10,2/6,2/8,2/9,2/10,3/6,4/6}
{
\draw[color=black] 
	(v-\i) -- (v-\j);
}
\end{tikzpicture}
}
\\[\baselineskip]
\subfigure[$|v_f(0)-v_f(1)|=2$.]
{
\begin{tikzpicture}
\tikzstyle{0-label}=[circle,fill=black!50,inner sep=3pt]
\tikzstyle{1-label}=[circle,fill=black,inner sep=3pt]

\foreach \i/\j/\k in {0/1.5/1,1/1.5/2,3/0/9}
{
  \node[1-label] (v-\k) at (\i,\j) {};
}

\foreach \i/\j/\k in {2/1.5/3,3/1.5/4,4/1.5/5,0/0/6,1/0/7}
{
  \node[0-label] (v-\k) at (\i,\j) {};
}

\foreach \i/\j/\k in {2/0/8,4/0/10}
{
  \draw
    (\i,\j) circle (4pt);
  \node[] (v-\k) at (\i,\j) {};
}

\foreach \i/\j in {1/7,1/8,1/9,1/10,2/6,2/8,2/9,2/10,3/9,5/9}
{
\draw[color=black] 
	(v-\i) -- (v-\j);
}
\end{tikzpicture}
}
\caption{$EBI(K_5 \times K_2) = \{ 0,1,2 \}$. We show only the $1$-edges.}
\label{fig:EBI-K5xK2}
\end{figure}

\begin {theo} \label{crownEBI}
  For any positive even integer $n$,
\begin{align*}
	EBI(K_n \times K_2)=\{0, 2, \dots, 2n-6, 2n-4\}
\end{align*}
and for any odd integer $n>3$,
\begin{align*}
	EBI(K_n \times K_2)=\{0, 1, \dots, 2n-9, 2n-8\}.
\end{align*}
\end {theo} 
 
\begin{proof}
 We label edges of $G=K_n\times K_2$ to find the maximum element of the edge-balanced index set. Vertices will be of two types, those which are incident to many more $0$-edges than $1$-edges, called \emph{dense}, and those which are incident to marginally more $1$-edges than $0$-edges, called \emph{sparse}. To find the maximum element of $EBI(G)$, our goal is to minimize the number of dense vertices and maximize the number of sparse vertices. Let $A$ and $B$ be the parts of $G$. 
 
\textbf{\emph{Case 1}}: Suppose $n$ is even. Choose vertices $u\in A$ and $v\in B$  where $u \not \sim v$. Set $D=\{u,v\}$ and $S=V(G)-D$. Label all edges incident to $u$ or $v$ by $0$. 
For $i=1, \dots, n-1$, let $u_i \in A \cap S$ and $v_i \in B \cap S$ with $u_i \not \sim v_i$. Define a cyclic order on the vertices $u_i \in A \cap S$ as follows: for $1 \leq i < n-1$, the \emph{succeeding} vertex of $u_i$ is $u_{i+1}$ and the succeeding vertex of $u_{n-1}$ is $u_1$. Similarly, define the \emph{next} vertex of $u_i$ to be $v_{i+1}$ for $1 \leq i < n-1$ and the next vertex of $u_{n-1}$ to be $v_1$. For each $i \in \{1, \ldots, n-1 \}$, let $U$ be the set containing $u_i$ and the succeeding $\frac{n-2}{2}$ vertices, so that $|U| = \frac{n}{2}$. Select the next vertex of each vertex in $U$ and label the edges between these vertices and $u_i$ by $1$. Label the remaining edges incident to $u_i$ by $0$, as in Figure~\ref{fig:n-even}.

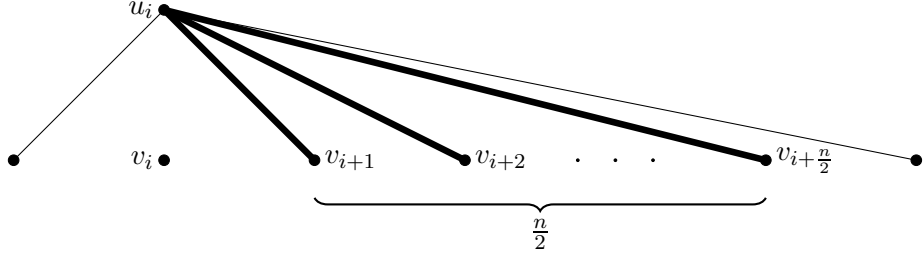
\begin{figure}[ht]
\centering
\begin{tikzpicture}
\draw [line width=2.5pt](0,2) -- (2,0);
\draw [line width=2.5pt](0,2) -- (4,0);
\draw [line width=2.5pt](0,2) -- (8,0);
\draw (0,2) -- (-2,0);
\draw (0,2) -- (10,0);

\coordinate[label=left:$u_i$] (1) at (0,2);
\coordinate[label=left:$v_i$] (4) at (0,0);
\coordinate[label=right:$v_{i+1}$] (5) at (2,0);
\coordinate[label=right:$v_{i+2}$] (6) at (4,0);
\coordinate[label=right:$v_{i+\frac{n}{2}}$] (7) at (8,0);

\filldraw [black]
(0,2) circle (2pt)
(-2,0) circle (2pt)
(0,0) circle (2pt)
(2,0) circle (2pt)
(4,0) circle (2pt)
(5.5,0) circle (.5pt)
(6,0) circle (.5pt)
(6.5,0) circle (.5pt)
(8,0) circle (2pt)
(10,0) circle (2pt);

\draw[thick,decorate,decoration={brace,amplitude=5pt}] 
(8,-.5) -- (2,-.5)
node[black,midway,below=4pt]{$\frac{n}{2}$}; 
\end{tikzpicture}
\caption{The case when $n$ is even. The heavy lines are the $1$-edges.}
\label{fig:n-even}
\end{figure}
 
 Under this edge labeling, the number of $0$-edges is $2(n-1)+(n-1)\frac{n-4}{2}$, the number of $1$-edges is $(n-1)\frac{n}{2}$, and the two quantities are identical. The vertices of $D$ are labeled $0$ and the vertices of $S$ are labeled $1$. Hence, $|v(0)-v(1)|=2n-4$. By Lemma~\ref{maxEBIregular} the edge-balanced index set contains no larger values.

 To attain the smaller values of $EBI(G)$, we switch $0$ and $1$ labels of pairs of edges incident to the same vertex. Each of these pairwise switches will decrease $|v(0)-v(1)|$ by $2$. The following is a list of such edge pairs:
\begin{itemize}
	\item $((u,v_2), (v_2,u_1))$, $((u,v_3), (v_3,u_2))$, \dots, $((u,v_{\frac{n}{2}}), (v_{\frac{n}{2}},u_{\frac{n-2}{2}}))$, switching the labels on $u_1, \dots, u_{\frac{n-2}{2}}$ from $1$ to $0$.
  \item $((v,u_1),(u_1,v_3))$, $((v,u_2),(u_2,v_4))$, \dots, $((v,u_{\frac{n-2}{2}}), (u_{\frac{n-2}{2}},v_{\frac{n+2}{2}}))$, switching the labels on $v_3, \dots, v_{\frac{n+2}{2}}$ from $1$ to $0$. 
\end{itemize}
By Lemma \ref{evenEBI}, we attain the first result.

\textbf{\emph{Case 2}}: Suppose $n$ is odd. Choose vertices $u,u' \in A$ and $v,v' \in B$ where $u \not\sim v$ and $u' \not\sim v'$. Set $D=\{u,u',v,v'\}$ and $S=V(G)-D$. Label all edges incident to a vertex of $D$ and a vertex of $S$ by $0$. For $i=1, \dots, n-2$, let $u_i \in A \cap S$ and $v_i \in B \cap S$ with $u_i \not \sim v_i$. Similar to the previous case, define a cyclic order on the vertices $u_i \in A \cap S$ as follows: for $1 \leq i < n-2$, the succeeding vertex of $u_i$ is $u_{i+1}$ and the succeeding vertex of $u_{n-2}$ is $u_1$. Likewise, define the next vertex of $u_i$ to be $v_{i+1}$ for $1 \leq i < n-2$ and the next vertex of $u_{n-2}$ to be $v_1$. For each $i \in \{1, \ldots, n-2 \}$, let $U$ be the set containing $u_i$ and the succeeding $\frac{n-1}{2}$ vertices, so that $|U| = \frac{n+1}{2}$. Select the next vertex of each vertex in $U$ and label the edges between these vertices and $u_i$ by $1$. Label the remaining edges incident to $u_i$ by $0$.
  
 Under this edge labeling, the number of $0$-edges is $2(n-2)+(n-2)\frac{n-3}{2}$, the number of $1$-edges is $(n-2)\frac{n+1}{2}$, and the two quantities are identical. Lastly, we label $(u,v')$ by $0$ and $(u',v)$ by $1$. The vertices of $D$ are labeled $0$ and the vertices of $S$ are labeled $1$. Hence, $|v(0)-v(1)|=2n-8$. 
 
 To show that $EBI(G)$ does not contain larger values, we argue as follows. For $|v(0)-v(1)|=2n$, $2n-1$, or $2n-2$ we sum the $1$-degrees of the vertices in a part that majorizes $1$-degree vertices and find the labeling is not edge-friendly. For $|v(0)-v(1)|=2n-3$, notice that we cannot have three unlabeled vertices and the rest labeled one. Thus, there is only one unlabeled vertex $x$ and only one $0$-labeled vertex $y$. If $x$ and $y$ are in the same part, then sum the $1$-degrees of vertices in the other part to find the labeling is not edge-friendly. If $x$ and $y$ are in opposite parts, sum the $1$-degrees of the vertices in the part containing $x$ to find the labeling is not edge-friendly. 
 
\emph{Subcase 1}: Suppose $|v(0)-v(1)|=2n-4$. If $v(0)=1$, then we argue as above. Thus we assume $v(0)=2$ and let $u$ and $v$ be the two vertices with label $0$. If $u$ and $v$ are in the same part, then summing the $1$-degrees of vertices in the other part shows the labeling is not edge-friendly. If $u$ and $v$ are in different parts, say $u\in A$ and $v\in B$, we find that for $n \geq 5$,
\begin{align*}
	e(1)=\sum_{w\in A, w\neq u}\deg_1(w) + \deg_1(u) \geq \frac{(n-1)(n+1)}{2} > \frac{n(n-1)}{2}+1, 
\end{align*}
and the labeling is not edge-friendly.
  
\emph{Subcase 2}: If $|v(0)-v(1)|=2n-5$, then $v(0)=2$ and some vertex must be unlabeled; say $u$ and $v$ are labeled $0$ and $x$ is unlabeled. If $u$ and $v$ are in the same part, we are done as described previously. Likewise, if $u$ and $v$ are in different parts, summing the $1$-degrees of the vertices in the part not containing $x$ (as in the previous subcase) results in an labeling that is not edge-friendly.

\emph{Subcase 3}: Suppose $|v(0)-v(1)|=2n-6$. If $v(0)=2$, then we argue as before. Suppose $v(0)=3$ and some part of $G$, say $A$, contains one or fewer vertices labeled $0$. Summing over the $1$-degrees of vertices in $A$ shows the labeling cannot be edge-friendly.
 
\emph{Subcase 4}: If $|v(0)-v(1)|=2n-7$, then $v(0)=3$ and some vertex must be unlabeled. Again, some part of $G$, say $A$, contains fewer vertices with label $0$, say $u\in A$ is labeled $0$. Summing over the $1$-degree of vertices in $A$, we find that for $n \geq 5$, 
\begin{align*}
	e(1)=\sum_{w\in A, w\neq u}\deg_1(w) + \deg_1(u) \geq \frac{n-1}{2}+\frac{(n-2)(n+1)}{2} = \frac{n^2-1}{2}-1, 
\end{align*}
and the labeling is not edge-friendly.
 
 To attain the smaller values of $EBI(G)$, we switch $0$ and $1$ labels of pairs of edges incident to the same vertex in such a way that every pairwise switch will decrease $|v(0)-v(1)|$ by $1$. The following is a list of such edge pairs:
\begin{itemize}
 \item $((u,v_2), (v_2,u_1))$, $((u',v_3), (v_3,u_1))$, $((u,v_3), (v_3,u_2))$, $((u',v_4), (v_4,u_2))$, \dots, $((u,v_{\frac{n-1}{2}}), (v_{\frac{n-1}{2}},u_{\frac{n-3}{2}}))$, $((u',v_{\frac{n+1}{2}}), (v_{\frac{n+1}{2}},u_{\frac{n-3}{2}}))$, switching the labels on each of $u_1, \dots, u_{\frac{n-3}{2}}$, first to unlabeled vertices and then to vertices labeled $0$.
 \item $((v,u_2), (u_2,v_1))$, $((v',u_3), (u_3,v_1))$, $((v,u_3), (u_3,v_2))$, $((v',u_4), (u_4,v_2))$, \dots, $((v,u_{\frac{n-3}{2}}), (u_{\frac{n-3}{2}},v_{\frac{n-5}{2}}))$, $((v',u_{\frac{n-1}{2}}), (u_{\frac{n-1}{2}},v_{\frac{n-5}{2}}))$, switching the labels on each of $v_1, \dots, v_{\frac{n-5}{2}}$, first to unlabeled vertices and then to vertices labeled $0$
\end{itemize}
Notice that after these switches, one part of $G$ contains $2+\frac{n-3}{2} = \frac{n+1}{2}$ vertices which are labeled $0$ and the other part of $G$ contains $2+\frac{n-5}{2} = \frac{n-1}{2}$ vertices which are labeled $0$, and there are no unlabeled vertices. Therefore, under this labeling, $|v(1)-v(0)|=0$. \qed
\end{proof}

\bibliographystyle{plain}

 \end{document}